\documentclass[12pt]{amsart}%
\usepackage{amsmath}
\usepackage{amsfonts}%
\usepackage{amssymb}%
\usepackage{graphicx}
\providecommand{\U}[1]{\protect\rule{.1in}{.1in}}
\newtheorem{theorem}{Theorem}

\newtheorem{corollary}[theorem]{Corollary}

\newtheorem{lemma}[theorem]{Lemma}

\newtheorem{proposition}[theorem]{Proposition}

\begin{document}

\title{On Neeman's gradient flows}
\author{Nolan R. Wallach}
\dedicatory{To Jim Lepowsky and Robert Wilson with admiration.}
\maketitle

\begin{abstract}
In his brilliant but sketchy paper on the strucure of quotient varieties of
affine actions of reductive algebraic groups over $\mathbb{C}$ Amnon Neeman
introduced a gradiant flow with remarkable properties. The purpose of this
paper is to study several applications of this flow. In particular we prove
that the cone on a Zariski closed subset of $\mathbb{P}^{n-1}(\mathbb{R})$ is
a deformation retract of $\mathbb{R}^{n}$. We also give an exposition of an
extension to real reductive algebraic group actions of Schwarz's excellent
explanation of Neeman's sketch of a proof of his deformation theorem. This
exposition precisely explains the use of Lojasiewicz gradient inequality. The
result described above for cones makes use of these ideas.

\end{abstract}

\section{Introduction}

The purpose of this note is to give an exposition of how an idea of Amnon
Neeman [N] (and Mumford) and results of Lojasiewicz [L] can be used to prove
some topological results for real projective varieties. For example, it is
proved that the affine cone on a Zariski closed subspace of real projective
space is a deformation retract of $\mathbb{R}^{n}$(see Theorem 11
in section 2). These ideas were applied to geometric invariant theory over
$\mathbb{C}$ by Neeman implying that if $G$ is a reductive group over
$\mathbb{C}$ acting on $\mathbb{C}^{n}$ and $K$ is a maximal compact subgroup
of $G$ (which we can assume is acts unitarily) and if $X$ is a $G$--invariant
subvariety of $\mathbb{C}^{n}$ then the Kempf-Ness set [KN] of $X$ is a strong
$K$--equivariant deformation retract of $X$. We give an argument for the
corresponding result over $\mathbb{R}$ (see also Richardson-Solovay [RS]).
There is a complete exposition of this aspect of the work in the paper of
Schwarz [S] (emphasizing the theory over $\mathbb{C}$). Anyone who has
attempted to read Neeman's paper ([N]), owes a debt of gratitude to the
careful exposition in [S]. [N] contains a weak form of the deformation theorem
in its first two sections. In sections four and later which contain the
more sophisticated topology Neeman mainly uses the weak form. Section three
contains the ideas mentioned above. In that section a sketch of the proof of
the deformation theorem is given on the basis of a \textquotedblleft
conjecture of Mumford\textquotedblright\ (3.1 in the paper) which he extends
by making another conjecture (3.5). In the introduction Neeman writes:

\smallskip

\textquotedblleft Now let us say something about Section 3. When I wrote the
paper it was a largely conjectural section, but now I know that both
Conjecture 3.1 and Conjecture 3.5 are true. Conjecture 3.5 is a special case
of an inequality due to Lojasiewicz, and Conjecture 3.1 can be proved from
Lojasiewicz's inequality using estimates similar to those in Section 3. I
chose not to rewrite the text, because at present I do not feel I could give
an adequate account of the proof of Conjecture 3.1. Although Lojasiewicz's
inequality is enough, a stronger inequality should be true; roughly speaking,
I conjecture that the correct value for $\varepsilon$ in Conjecture 3.5 is 1/2
(see remark 3.7). For this reason I feel the appendix is still important; it
contains evidence for my new conjectures. If I rewrote Section 3 to
incorporate my new conjectures, the new section would be too long, and largely
unconnected with the rest of the paper.\textquotedblright

\smallskip

In this paper we expand a bit on the exposition of [S] and prove a stronger
form of \textquotedblleft Conjecture 3.1\textquotedblright\ (following
Neeman's suggestion). Neeman also conjectured that the correct $\varepsilon$ is
$\frac{1}{2}$. Neeman gives a
sketch of an argument in the case of tori (alluded to in the quote) which we expand in the last section. We observe that his argument doesn't use the
Lojasiewicz theory to get the stronger result.

The result of Lojasiewicz involves mathematics outside of the usual universe
of researchers in the theory of algebraic groups involving the study of real
algebraic (and analytic) inequalities initiated in the Tarski-Seidenberg
theorem (c.f. [H]) and expanded on in Lojasiewicz in his development of real
analytic geometry ([L]). Since this theory is also far away from my expertise,
I show, in the last section, that some of the ideas that only involve
freshman calculus can be used to prove useful weaker results.

\section{Some gradient systems}

Let $\phi\in\mathbb{R}[x_{1},...,x_{n}]$ be a polynomial that is homogeneous
of degree $m$ such that $\phi(x)\geq0$ for all $x\in\mathbb{R}^{n}$. We
consider the gradient system%
\[
\frac{dx}{dt}=-\nabla\phi(x)
\]
relative to the usual inner product on $\mathbb{R}^{n}$, $\left\langle
x,y\right\rangle =\sum x_{i}y_{i}$. Where, as usual,
\[
\nabla\phi(x)=\sum\frac{\partial\phi}{\partial x_{i}}e_{i}%
\]
with $\{e_{1},...,e_{n}\}$ the standard orthonormal basis. Then%
\[
\left\langle \nabla\phi(x),x\right\rangle =m\phi(x).\overset{}{(\ast)}%
\]
So, if we denote by $F(t,x)$ the solution to the system for $t$ near $t=0$
with $F(t,0)=x$ then%
\[
\frac{d}{dt}\left\langle F(t,x),F(t,x)\right\rangle =-2\left\langle \nabla
\phi(F(t,x)),F(t,x)\right\rangle =-2m\phi(F(t,x))\leq0.
\]
This implies

\begin{lemma}
$\left\Vert F(t,x)\right\Vert \leq\left\Vert x\right\Vert $ if $F(s,x)$ is
defined for $0\leq s\leq t$.
\end{lemma}

We therefore have

\begin{lemma}
$F(t,x)$ is defined for all $t\geq0$, $x\in\mathbb{R}^{n}$ and smooth in
$(t,x)$.
\end{lemma}

\begin{proof}
Assume that $F(t,x)$ is defined for $0\leq t<t_{o}$. Let $\{t_{j}\}$ be a
sequence in $[0,t_{o})$ with $\lim_{j\rightarrow\infty}t_{j}=t_{o}$. Then
since $\left\Vert F(t_{j},x)\right\Vert \leq\left\Vert x\right\Vert $ there is
an infinite subsequence $\{t_{j_{k}}\}$ such that $\left\{  F(t_{j_{k}%
},x)\right\}  $ converges to $x_{o}$. Let $\varepsilon>0$ be such that
$F(s,y)$ is defined and smooth on $|s|<\varepsilon$ and $(-\varepsilon
,\varepsilon)\times B_{\varepsilon}(x_{o})$ ($B_{r}(y)$ is the usual Euclidean
$r$--ball with center $y$). There exists $N$ such that if $k\geq N$ then
$|t_{j_{k}}-t_{0}|<\varepsilon$ and $\left\Vert F(t_{j_{k}},x)-x_{o}%
\right\Vert <\varepsilon$. Fix $k\geq N$. Then $t_{j_{k}}=t_{o}-s$ with
$|s|<\varepsilon$ and $\left\Vert F(t_{o}-s,x)-x_{o}\right\Vert <\varepsilon$.
Thus if $\delta=|\varepsilon-|s||$ and $|u|<\varepsilon$ then $F(s+u,F(t_{o}-s,x))$ is defined.
Hence $F(t_{o}+u,x)$ is defined for $\left\vert u\right\vert <\delta$ and
given by $F(s+u,F(t_{o}-s,x))$.
\end{proof}

The formula $(\ast)$ combined with the Schwarz inequality implies

\begin{lemma}
$\left\Vert \nabla\phi(x)\right\Vert \left\Vert x\right\Vert \geq m\phi(x) $.
Thus if $\left\Vert x\right\Vert \leq r$ then%
\[
\left\Vert \nabla\phi(x)\right\Vert \geq\frac{m}{r}\phi(x).
\]

\end{lemma}

The Lojasiewicz gradient inequality [L] implies the following improvement of
the equality in the above Lemma.

\begin{theorem}
Assume that $m>1$. There exists $0<\varepsilon\leq\frac{1}{m-1}$ and $C>0$
both depending only on $\phi$ such that for all $x\in\mathbb{R}^{n}$%
\[
\left\Vert \nabla\phi(x)\right\Vert ^{1+\varepsilon}\left\Vert x\right\Vert
^{1-(m-1)\varepsilon}\geq C\phi(x).
\]

\end{theorem}

To see this we recall the Lojasiewicz inequality

\begin{theorem}
If $\psi$ is a real analytic function on an open subset, $U$, of
$\mathbb{R}^{n}$ and if $x_{o}\in U$ then there exist $C>0$, $\varepsilon>0$
and $r>0 $ such that $B_{r}(x_{o})=\{x\in\mathbb{R}^{n}|\left\Vert
x-x_{o}\right\Vert <r\}\subset U$ and
\[
\left\Vert \nabla\psi(x)\right\Vert ^{1+\varepsilon}\geq C|\psi(x)-\psi
(x_{o})|
\]
if $x\in B_{r}(x_{o})$.
\end{theorem}

To prove the asserted implication we note since $\phi(0)=0$ there exist
$\varepsilon$ and $r$ as in the theorem above so that%
\[
\left\Vert \nabla\phi(x)\right\Vert ^{1+\varepsilon}\geq C|\phi(x)|,x\in
B_{r}(0).
\]
If $\varepsilon>\frac{1}{m-1}$ we argue that we may replace $\varepsilon$ with
any $0<\delta\leq\frac{1}{m-1}$. Since $\nabla\phi(0)=0$ we can choose $s\leq
r$ such that if $\left\Vert x\right\Vert <s$ then $\left\Vert \nabla
\phi(x)\right\Vert \leq1$ hence if $\left\Vert x\right\Vert <s,\left\Vert
\nabla\phi(x)\right\Vert ^{1+\delta}\geq\left\Vert \nabla\phi(x)\right\Vert
^{1+\varepsilon}.$Thus we may assume $0<\varepsilon\leq\frac{1}{m-1}$. We now
may scale in $x$ (using the fact that $\nabla\phi$ is homogeneous of degree
$m-1$) to see that with a different constant $C$ we have%
\[
\left\Vert \nabla\phi(x)\right\Vert ^{1+\varepsilon}\geq C|\phi(x)|,x\in
\overline{B_{1}(0)}.
\]
Thus if $\left\Vert x\right\Vert =1$ we have%
\[
\left\Vert \nabla\phi(x)\right\Vert ^{1+\varepsilon}\left\Vert x\right\Vert
^{1-\left(  m-1\right)  \varepsilon}\geq C|\phi(x)|.
\]
Noting that the homogeneity of the left hand side is 
\[1+\varepsilon)(m-1)+1-(m-1)\varepsilon=m\] 
the theorem now follows.
Since $\phi$ is homogeneous of degree $m$. One is tempted, on the basis of
homogeneity, to think that $\varepsilon=\frac{1}{m-1}$ would be the correct
choice in the theorem above. This is related to Neeman's remark 3.7 as mentioned in the introduction.

\section{The Neeman flow (as explained by Gerry Schwarz)}

We use the notation of the previous section. We take $\varepsilon$ and $C$ as
above (but note that one can very simply get the estimate in the theorem with
$\varepsilon=0)$. If we write $F$ for $F(t,X)$ and $H(t)=\phi(F(t,x))$ then we
have%
\[
H^{\prime}(t)=-d\phi(F)\left(  \nabla\phi(F)\right)  =-\left\Vert \nabla
\phi(F)\right\Vert ^{2}.
\]
If $t\geq0$ and $\left\Vert x\right\Vert \leq r$%
\[
\left\Vert \nabla\phi(F)\right\Vert ^{1+\varepsilon}\left\Vert F\right\Vert
^{1-(m-1)\varepsilon}\geq C\phi(F).
\]
Thus
\[
\left\Vert \nabla\phi(F)\right\Vert ^{1+\varepsilon}\geq\frac{C}%
{r^{1-(m-1)\varepsilon}}\phi(F).
\]
Hence%
\[
\left\Vert \nabla\phi(F)\right\Vert ^{2}\geq\left(  \frac{C}%
{r^{1-(m-1)\varepsilon}}\right)  ^{\frac{2}{1+\varepsilon}}\phi(F)^{\frac
{2}{1+\varepsilon}}.
\]
Thus%
\[
|H^{\prime}(t)|\geq\frac{1}{2}\left(  \frac{C}{r^{1-(m-1)\varepsilon}}\right)
^{\frac{2}{1+\varepsilon}}\phi(F)^{\frac{2}{1+\varepsilon}}=C_{1}%
(r)H(t)^{\frac{2}{1+\varepsilon}}.
\]
This yields (since $H^{\prime}(t)\leq0$)%
\[
-H^{\prime}(t)\geq C_{1}(r)H(t)^{\frac{2}{1+\varepsilon}}\text{.}%
\]
Thus
\[
\frac{d}{dt}H(t)^{-\frac{1}{1+\varepsilon}}=-\frac{H^{\prime}(t)}%
{H(t)^{\frac{2}{1+\varepsilon}}}\geq C_{1}(r)
\]
we conclude that if $t>0$ then%
\[
H(t)^{-\frac{1}{1+\varepsilon}}\geq C_{1}(r)t.
\]
Inverting we have%
\[
H(t)\leq C_{2}(r)t^{-(1+\varepsilon)}%
\]
with $C_{2}(r)=C_{1}(r)^{-(1+\varepsilon)}$. The result of Lojasiewicz gains
us the $\varepsilon>0$. The key aspect of this inequality is that the the only
dependence is on $r$ so it is true for any $F(t,x)$ with $\left\Vert
x\right\Vert \leq r$ and $t>0$. In many cases the easy case $\varepsilon=0$ is
sufficient. We now show how the $\varepsilon>0$ leads to an important result
(the argument is modeled on the exposition of  G.
Schwarz [S]).

We note that the above inequality implies that if $f(t)=t^{1+\delta}$ with
$0<\delta<\varepsilon$ then for $t>0$%
\[
0<H(t)f^{\prime}(t)\leq C_{2}(r)(1+\delta)t^{-1-(\varepsilon-\delta)}.
\]
Let $0<t<s$ then%
\[
H(s)f(s)-H(t)f(t)=\int_{t}^{s}\frac{d}{du}(H(u)f(u))du=
\]%
\[
\int_{t}^{s}H(u)f^{\prime}(u)du+\int_{t}^{s}H^{\prime}(u)f(u)du.
\]
Thus%
\[
-\int_{t}^{s}H^{\prime}(u)f(u)du=\int_{t}^{s}H(u)f^{\prime}%
(u)du+H(t)f(t)-H(s)f(s).
\]
We also note that \[
0\leq H(s)f(s)\leq C_{2}(r)s^{-(1+\varepsilon)}s^{1+\delta}=C_{2}%
(r)s^{-(\varepsilon-\delta)}.
\]
Since $|H^{\prime}(u)|=-H^{\prime}(u)$ this implies
\[
\lim_{s\rightarrow+\infty}\int_{t}^{s}\left\vert H^{\prime}(u)\right\vert
f(u)du=\int_{t}^{\infty}H(u)f^{\prime}(u)du+H(t))f(t) < \infty.
\]
Thus $\sqrt{\left\vert H^{\prime}(u)\right\vert f(u)}$ is in $L^{2}%
([t,+\infty))$ for all $t>0$ and so
\[
\left\vert H^{\prime}(u)\right\vert =\sqrt{\left\vert H^{\prime}(u)\right\vert
f(u)}u^{-\frac{(1+\delta)}{2}}\in L^{1}([t,+\infty)).
\]
All estimates are uniform for $\left\Vert x\right\Vert \leq r<\infty$ so we
have proved:

\begin{theorem}
If $t>0$ then%
\[
\int_{t}^{+\infty}\left\Vert \frac{d}{du}F(u,x)\right\Vert du
\]
converges uniformly for $\left\Vert x\right\Vert \leq r$.
\end{theorem}

This result implies that if $t\geq0$ then
\[
\int_{t}^{\infty}\frac{d}{du}F(u,x)du
\]
converges absolutely and uniformly for $\left\Vert x\right\Vert \leq r<\infty
$. Noting that if $s>t$ then%
\[
\int_{t}^{s}\frac{d}{du}F(u,x)du=F(s,x)-F(t,x)
\]
we have for $t>0$%
\[
\lim_{s\rightarrow\infty}F(s,x)=\int_{t}^{\infty}\frac{d}{du}F(u,x)du+F(t,x).
\]
So if we set $U(t,x)=F(\frac{t}{1-t},x)$ and define $U(1,x)$ by the limit
above then $U:[0,1]\times\mathbb{R}^{n}\rightarrow\mathbb{R}^{n}$ is
continuous and
\[
\nabla\phi(x)=0\Longleftrightarrow\phi(x)=0
\]
( Lemma 3  and the fact that $0$ is a minimum for $\phi$) we have proved

\begin{theorem}
\label{retract}$U:[0,1]\times\mathbb{R}^{n}\rightarrow\mathbb{R}^{n}$ defines
a strong deformation retraction of $\mathbb{R}^{n}$ onto $Y=\{x\in
\mathbb{R}^{n}|\phi(x)=0\}$.
\end{theorem}

\begin{proof}
We note since $\nabla\phi(y)=0$ if $y\in Y$ then $F(t,y)=y$ for all $y\in Y $.
Thus $U(0,x)=x$ all $x\in\mathbb{R}^{n}$, $U(t,y)=y$ all $0\leq t\leq1 $ and
all $y\in Y$ and since%
\[
\lim_{t\rightarrow+\infty}\phi(F(t,x))=0
\]
we have $U(1,\mathbb{R}^{n})=Y$.
\end{proof}
A deformation retration of a topological space $X$ onto a closed subspace $Y$ is
a continuous map $U : [0,1] \times X \rightarrow X$ such that $U(1,X)=X$  and $U(t,y)=y$
for all $y \in Y$ and $t \in [0,1]$.

We now derive a few corollaries to this result. The first is obvious.

\begin{corollary}
If $X\subset\mathbb{R}^{n}$ is a closed subset such that $F(t,X)\subset X$ for
all $t\geq0$ then $Y\cap X$ is a strong deformation retraction of $X$.
\end{corollary}

\begin{corollary}
\label{subinv}Let $K$ be a compact subgroup of $GL(n,\mathbb{R})$ and assume
that $\phi(kx)=\phi(x)$ for $k\in K,x\in\mathbb{R}^{n}$. If $X$ is as above
and invariant under $K$ then the strong retraction in the previous corollary
is $K$ equivariant.
\end{corollary}

\begin{proof}
We note that the $K$--invariance of $\phi$ implies that $\nabla\phi
(kx)=k\nabla\phi(x)$ for $k\in K,x\in\mathbb{R}^{n}$. Thus
\[
\frac{d}{dt}k^{-1}F(t,kx)=-k^{-1}\nabla\phi(F(t,kx))-\nabla\phi(k^{-1}F(t,kx))
\]
and since
\[
k^{-1}F(0,kx)=x
\]
the uniqueness theorem implies that%
\[
k^{-1}F(t,kx)=F(t,x).
\]

\end{proof}

We now assume that $Y\subset\mathbb{R}^{n}$ is the locus of zeros of
homogeneous polynomials $f_{1},...,f_{m}$ with $\deg f_{i}=r_{i}$. We set
$r=\operatorname{lcm}(r_{1},...,r_{m})$ and
\[
\phi(x)=\sum_{i=1}^{m}(f_{i}^{\frac{r}{r_{i}}})^{2}.
\]
Then $Y=\{x\in\mathbb{R}^{n}|\phi(x)=0\}$. Let $F(t,x)$ be as above for this
choice of $\phi$. Then we can apply the Corollaries to this case.

Finally, let $K$ be a compact subgroup of $GL(n,\mathbb{R})$ and $KY\subset Y
$ with $Y$ the zero locus of $f_{i}$ for $f_{i}$ as above.

\begin{lemma}
Define $\phi_{K}(x)=\int_{K}\phi(kx)dk$ then $\phi_{K}$ is a homogeneous
polynomial of degree $2r,$ $\phi_{K}(x)\geq0$ all $x\in\mathbb{R}^{n}$ and
$Y=\{x\in\mathbb{R}^{n}|\phi_{K}(x)=0\}$.
\end{lemma}

\begin{proof}
We note that%
\[
\int_{K}\phi(kx)dk=\sum_{i=1}^{m}\int_{K}(f_{i}(kx))^{\frac{2r}{r_{i}}}dk.
\]
Thus since each integrand is non-negative if $\phi_{K}(x)=0$ then we have for
all $i$%
\[
\int_{K}(f_{i}(kx))^{\frac{2r}{r_{i}}}dk=0
\]
and hence $f_{i}(kx)=0$ for all $k$ and $i$. Hence $x\in Y$. The lemma is now obvious.
\end{proof}

Combining this with the above Corollary we have

\begin{theorem}
If $X\subset\mathbb{P}^{n-1}(\mathbb{R})$ is a $K$ invariant Zariski closed
then there exists a $K$--equivariant strong deformation retract of
$\mathbb{R}^{n}$ to the cone on $X$ in $\mathbb{R}^{n}$.
\end{theorem}

\section{Neeman's theorem.}

We now look at the  main example for which the conditions of the above corollaries are
satisfied.

Let $G$ be a real algebraic subgroup of $GL(n,\mathbb{R})$ invariant under
transpose and let $K=G\cap O(n)$. Let for $x\in\mathbb{R}^{n},X\in
\mathfrak{g}=Lie(G)$
\[
f_{x}(X)=\left\langle Xx,x\right\rangle
\]
then $f_{x}\in\mathfrak{g}^{\ast}$. On $\mathfrak{g}^{\ast}$ we put the inner
product dual to $(X,Y)=\mathrm{tr}(XY^{\ast})$ (here $Y^{\ast}$ is just the
transpose of $Y$). Then we take
\[
\phi(x)=\left\Vert f_{x}\right\Vert ^{2}.
\]
Looking upon $\mathbb{R}^{n}$ as $n\times1$ matrices we have%
\[
f_{x}(X)=\mathrm{tr}(Xxx^{\ast}).
\]
Hence $f_{x}$~$(X)$ is the inner product of $X$ with $P_{\mathfrak{g}%
}(xx^{\ast})$ where $P_{\mathfrak{g}}$ is the orthogonal projection of
$M_{n}(\mathbb{R})$ onto $\mathfrak{g}$. So%
\[
\phi(x)=\mathrm{tr}\left(  P_{\mathfrak{g}}(xx^{\ast})^{2}\right)  .
\]

We now compute the gradient of $\phi$%
\[
d\phi_{x}(v)=2\mathrm{tr}P_{\mathfrak{g}}(vx^{\ast}+xv^{\ast})P_{\mathfrak{g}%
}xx^{\ast})=
\]%
\[
2\mathrm{tr}((vx^{\ast}+xv^{\ast})P_{\mathfrak{g}}xx^{\ast})=2\left\langle
v,P_{\mathfrak{g}}(xx^{\ast})x\right\rangle +2\left\langle x,P_{\mathfrak{g}%
}(xx^{\ast})v\right\rangle =
\]%
\[
4\left\langle v,P_{\mathfrak{g}}(xx^{\ast})x\right\rangle
\]
since $P_{\mathfrak{g}}(xx^{\ast})^{\ast}=P_{\mathfrak{g}}(xx^{\ast})$. Thus%
\[
\nabla\phi(x)=4P_{\mathfrak{g}}(xx^{\ast})x\in T_{x}(Gx).
\]
This implies that $F(t,x)\in Gx$ for all $t\geq0$.

To put this in context we recall the Kempf-Ness theorem over $\mathbb{R}$.
Then $v\in%
\mathbb{R}
^{n}$ will be said to be \emph{critical} if $\left\langle Xv,v\right\rangle
=0$ for all $X\in\mathfrak{g}=Lie(G)$. We note that this is the same as saying
that $\left\langle Xv,v\right\rangle =0$ for all $X\in\mathfrak{p}%
=\{Y\in\mathfrak{g}|Y^{\ast}=Y\}$,$.$ Here is the Kempf-Ness theorem in this
context (the topological assertions are for the subspace topology in
$\mathbb{R}^{n}$).

\begin{theorem}
\label{Kempf-Ness}Let $G,K$ be as above. Let $v\in%
\mathbb{R}
^{n}$.

1. $v$ is critical if and only if $\left\Vert gv\right\Vert \geq\left\Vert
v\right\Vert $for all $g\in G$.

2. If $v$ is critical and $X\in\mathfrak{p}$ is such that $\left\Vert
e^{X}v\right\Vert =\left\Vert v\right\Vert $then $Xv=0$. If $w\in Gv$ is such
that $\left\Vert v\right\Vert =\left\Vert w\right\Vert $then $w\in Kv.$

3. If $Gv$ is closed then there exists a critical element in $Gv$.

4. If $v$ is critical then $Gv$ is closed.
\end{theorem}

We set $Crit_{G}(\mathbb{R}^{n})$ equal to the real algebraic variety of
critical elements. We note that $Crit_{G}(\mathbb{R}^{n})$ is the zero set of
$\phi(x)=\mathrm{tr}P_{\mathfrak{g}}(xx^{\ast})^{2}$.

We can now state the theorem of Neeman over $\mathbb{R}$.

\begin{theorem}
Let $X$ be a $G$--invariant closed subset of $\mathbb{R}^{n}$ then $X\cap
Crit(\mathbb{R}^{n})$ is a strong $K$--equivariant deformation retract of $X$.
\end{theorem}

\begin{proof}
We note that $\phi(x)=\mathrm{tr}P_{\mathfrak{g}}(xx^{\ast})^{2}$ is
$K$--invariant and $F(t,x)\in Gx$ thus any $G$--invariant subset of
$\mathbb{R}^{n}$ is invariant under the flow. The theorem follows from
Corollary \ref{subinv}.
\end{proof}

In the course of our proof of this version of the Kempf-Ness theorem we proved
an auxiliary result (see [W], 3.6.2 ). Let $G_{\mathbb{C}}$ be the Zariski
closure of $G$ in $GL(n,\mathbb{C})$ then $G_{\mathbb{C}}$ is invariant under
adjoint and hence is reductive. Let $L=G_{\mathbb{C}}\cap U(n)$ then $L$ is a
maximal compact subgroup of $G_{\mathbb{C}}$ and $L\cap G=K$. The Kempf-Ness
theorem (in the complex case) implies that if $v\in\mathbb{C}^{n}$ is
$G_{\mathbb{C}}$--critical then $G_{\mathbb{C}}v\cap Crit(\mathbb{C}^{n})=Uv$.
The following result was proved

\begin{proposition}
If $v\in$ $\mathbb{R}^{n}$ is $G$--critical then it is $G_{\mathbb{C}}$
critical and $G_{\mathbb{C}}v\cap\mathbb{R}^{n}$ is a finite union of open
$G$--orbits (hence closed).
\end{proposition}

We note that this shows that $4.$in the Kempf-Ness theorem over $\mathbb{C}$
implies $4.$in the theorem over $\mathbb{R}$ (the rest is just calculus).

\begin{corollary}
If $v\in$ $\mathbb{R}^{n}$ is $G$--critical then ~$Lv\cap\mathbb{R}^{n}%
=Kv_{1}\cup\cdots\cup Kv_{r}$ a finite number of $K$--orbits.
\end{corollary}

\begin{proof}
Since $G_{\mathbb{C}}v\cap\mathbb{R}^{n}$ is closed, the above proposition and
3. in the Kempf-Ness theorem imply that $G_{\mathbb{C}}v\cap\mathbb{R}^{n}=$
$\cup_{j=1}^{r}Gv_{j}$ with $v_{j}$ critical in $\mathbb{R}^{n}$. Since
$G_{\mathbb{C}}v\cap Crit(\mathbb{C}^{n})=Lv$, and $Crit(\mathbb{C}^{n}%
)\cap\mathbb{R}^{n}=Crit(\mathbb{R}^{n})$ we have%
\[
Lv\cap\mathbb{R}^{n}=\left(  \cup_{j=1}^{r}Gv_{j}\right)  \cap Crit(\mathbb{R}%
^{n})=
\]%
\[
\cup_{j=1}^{r}\left(  Gv_{j}\cap Crit(\mathbb{R}^{n})\right)  =Kv_{1}%
\cup\cdots\cup Kv_{r}\text{.}%
\]

\end{proof}

The $r$ in the statement can be larger than $1$. This is the reason why the
next section is over $\mathbb{C}$.

\section{An elementary result}

We retain the notation of the previous section. In this section we explain how
the elementary estimate (that only uses Freshman calculus)
\[
\phi(F(t,x))\leq\frac{C(\left\Vert x\right\Vert )}{t}%
\]
for $t>0$ can prove a useful weakening of Neeman's theorem for actions of
connected reductive algebraic groups over $\mathbb{C}$. Let $G\subset
GL(n,\mathbb{C})$ be Zariski closed and invariant under adjoint. Let $K$ be
the intersection of $G$ with $U(n).$ We look upon $\mathbb{C}^{n}$ as
$\mathbb{R}^{2n}=\mathbb{R}^{n}\oplus i\mathbb{R}^{n}$ and $G$ as a real
algebraic group. Thus $K$ is also the intersection of $G$ with $O(2n)$. In
this context if $v\in\mathbb{R}^{2n}$ then $\overline{Gv}$ contains a unique
closed orbit and $\overline{Gv}\cap Crit(\mathbb{R}^{2n})$ is a single
$K$--orbit. We also note that $F(t,kv)=kF(t,v)$. Thus $F$ induces a flow on
$\mathbb{R}^{2n}/K$, which we denote by $H(t,Kx)$.

We note

\begin{theorem}
Let $v\in\mathbb{C}^{n}$ then $\lim_{t\rightarrow+\infty}H(t,Kv)=\overline
{Gv}\cap Crit(\mathbb{R}^{2n})=Ku$.
\end{theorem}

\begin{proof}
The above estimate implies that%
\[
\lim_{t\rightarrow+\infty}\phi(F(t,v))=0.
\]
We have also seen that if $t>0$, then $\left\Vert F(t,v)\right\Vert
\leq\left\Vert v\right\Vert .$ Let $\{t_{j}\}$ be a sequence in $\mathbb{R}%
_{>0}$ such that $\lim_{j\rightarrow\infty}t_{j}=+\infty$. The sequence
$\left\{  F(t_{j},v)\right\}  $ is bounded. Let $F(t_{j_{k}},v)$ be a
convergent subsequence. Then $\lim_{k\rightarrow\infty}F(t_{j_{k}},v)=u\in$
$\overline{Gv}$ and $\phi(u)=0$. Thus $Ku=\overline{Gv}\cap Crit(\mathbb{R}%
^{2n})$. Thus every convergent subsequence of $\left\{  H(t_{j},Kv)\right\}  $
converges to $Ku$. This implies the theorem.
\end{proof}

\section{Neeman's argument for Tori}

As indicated in the introduction Neeman conjectured that in the context of
Section 4 (there $\phi $ is homogeneous of degree 4) there should exist $C>0$
such that for all $x$%
\[
C\left\Vert \nabla \phi (x)\right\Vert ^{\frac{4}{3}}\geq \phi (x). 
\]

As evidence for this assertion he gave a sketch of a proof for the case when 
$G$ (in that section is commutative). We will devote this section to filling
out his brilliant proof this case. We first set up the general question. Let 
$G$ be a closed subgroup of $GL(N,\mathbb{R})$ such that $G$ is invariant
under adjoint. Let $\mathfrak{p}=\{X\in Lie(G)|X^{\ast }=X\}$. We have seen
that if $P$ is the orthogonal projection of $M_{N}(\mathbb{R})$ onto $%
\mathfrak{p}$ (here we are using the inner product $\mathrm{tr}XY^{\ast }$)
then $\phi (x)=tr\left( P(xx^{\ast })\right) ^{2}$ (here we look upon $x$ as
an $N\times 1$ column). Now if $X_{1},...,X_{n}$ is an orthonormal basis of $%
\mathfrak{p}$ then%
\[
P(xx^{\ast })=\sum_{i}\mathrm{tr}(X_{i}xx^{\ast })X_{i}=\sum_{i}\left\langle
X_{i}x,x\right\rangle X_{i}
\]%
and 
\[
\phi (x)=\sum_{i}\left\langle X_{i}x,x\right\rangle ^{2}.
\]%
We also note that 
\[
\nabla \phi (x)=4\sum_{i}\left\langle X_{i}x,x\right\rangle X_{i}x.
\]%
Hence 
\[
\left\Vert \nabla \phi (x)\right\Vert ^{2}=\sum_{i,j}\left\langle
X_{i}v,v\right\rangle \left\langle X_{j}v,v\right\rangle \left\langle
X_{i}v,X_{j}v\right\rangle .
\]%
Thus the theorem below implies the desired result for the case when $G$ is
abelian. The following lemma plays an important role in the proof of the
theorem and since it may not be well known so we include a proof before
embarking on the proof of the theorem.

Let $\left( V,\left\langle ...,...\right\rangle \right) $ be a finite
dimensional inner product space over $\mathbb{R}$.

\begin{lemma}
Let $v_{1},...,v_{n}\in V$ spanning an $m$--dimensional vector space. Then
there exists $A=\left[ a_{ij}\right] _{1\leq i,j\leq n}$ an orthogonal
matrix over $\mathbb{R}$ and $c_{1},...,c_{k}$ in $\mathbb{R}_{>0}$ such
that if $z_{i}=\sum_{j}a_{ij}v_{j}$ then $z_{j}=0$ for $j>m$ and%
\[
\left\langle z_{i},z_{j}\right\rangle =\delta _{ij}c_{i},1\leq i,j\leq m. 
\]
\end{lemma}

\begin{proof}
After permuting the $v_{j}$ we may assume that $v_{1},...,v_{k}$ are
linearly independent. Let 
\[
v_{m+j}=\sum_{i=1}^{m}x_{j,i}v_{i}\text{.}
\]%
Let $X$ be the $n-m$ by $m$ matrix with entries $x_{ij}$. We form the block
matrix%
\[
B=\left[ 
\begin{array}{cc}
I_{m} & 0 \\ 
X & -I_{n-m}%
\end{array}%
\right] =\left[ b_{ij}\right] 
\]%
with $I_{r}$ the $r\times r$ identity matrix. Then $\sum_{j}b_{ij}v_{j}=v_{i}
$ for $i\leq m$ and $\sum_{j}b_{ij}v_{j}=0$ for $i>m$. Using the Iwasawa
decomposition for $GL(n,\mathbb{R})$ (i.e. Gram-Schmidt) we can write%
\[
B=uak
\]%
with $u$ upper triangular with $1$'s on the main diagonal, $a$ diagonal with
positive diagonal entries $a_{1},...,a_{n}$ and $k\in O(n)$. We have%
\[
B\left[ 
\begin{array}{c}
v_{1} \\ 
\vdots  \\ 
v_{n}%
\end{array}%
\right] =\left[ 
\begin{array}{c}
\sum b_{1j}v_{j} \\ 
\vdots  \\ 
\sum b_{mj}v_{j}%
\end{array}%
\right] =\left[ 
\begin{array}{c}
v_{1} \\ 
\vdots  \\ 
v_{m} \\ 
0 \\ 
\vdots  \\ 
0%
\end{array}%
\right] .
\]%
So 
\[
ak\left[ 
\begin{array}{c}
v_{1} \\ 
\vdots  \\ 
v_{n}%
\end{array}%
\right] =u^{-1}B\left[ 
\begin{array}{c}
v_{1} \\ 
\vdots  \\ 
v_{n}%
\end{array}%
\right] =u^{-1}\left[ 
\begin{array}{c}
v_{1} \\ 
\vdots  \\ 
v_{m} \\ 
0 \\ 
\vdots  \\ 
0%
\end{array}%
\right] =\left[ 
\begin{array}{c}
w_{1} \\ 
\vdots  \\ 
w_{m} \\ 
0 \\ 
\vdots  \\ 
0%
\end{array}%
\right] 
\]%
with $w_{1},...,w_{m}$ linearly independent. Now apply $a^{-1}$ and have%
\[
k\left[ 
\begin{array}{c}
v_{1} \\ 
\vdots  \\ 
v_{n}%
\end{array}%
\right] =\left[ 
\begin{array}{c}
a_{1}^{-1}w_{1} \\ 
\vdots  \\ 
a_{m}^{-1}w_{m} \\ 
0 \\ 
\vdots  \\ 
0%
\end{array}%
\right] =.\left[ 
\begin{array}{c}
t_{1} \\ 
\vdots  \\ 
t_{m} \\ 
0 \\ 
\vdots  \\ 
0%
\end{array}%
\right] .
\]%
Finally, we choose an orthogonal $m\times m$ matrix $T$ that diagonalizes
the form%
\[
\sum_{1\leq i,j\leq m}x_{i}\left\langle t_{i},t_{j}\right\rangle x_{j}.
\]%
Setting 
\[
S=\left[ 
\begin{array}{cc}
T & 0 \\ 
0 & I%
\end{array}%
\right] 
\]%
then $A=Sk$ is the desired orthogonal transformation.
\end{proof}

\begin{corollary}
Let $X_{1},...,X_{n}\in End(V)$ and $v\in V$. Suppose that the span of $%
\left\{ X_{i}v\right\} $ has dimension $m$. Then there exists $A=[a_{ij}]\in
O(n)$ such that if $Z_{i}=\sum a_{ij}X_{j}$ then $Z_{i}v=0$ for $i>m$ and $%
\left\langle Z_{i}v,Z_{j}v\right\rangle =\delta _{ij}c_{i}$ with $c_{i}>0$
for $i\leq m$.
\end{corollary}

\begin{proof}
Apply the above lemma to $v_{i}=X_{i}v$, $i=1,...,n$.
\end{proof}

We note that if $X_{1},...,X_{n}$ are self adjoint elements of $End(V)$ and $%
\phi (x)=\sum_{i=1}^{n}\left\langle X_{i}v,v\right\rangle ^{2}$ then $\nabla
\phi (x)=4\sum_{i=1}^{n}\left\langle X_{i}v,v\right\rangle X_{i}v$. In this
case the homogeneity is $m=4$ and thus the suggested strong form of the
inequality is 
\[
C\left\Vert \nabla \phi (x)\right\Vert ^{1+\frac{1}{3}}\geq \phi (x).
\]%
The following theorem of Neeman proves this result if the $X_{i}$ mutually
commute. We include a detailed proof following Neeman's sketch since this
result is so suggestive. We also make clear where the commutivity assumption
is used (exactly one step). In the proof we will use the obvious identity 
\[
\left\Vert \sum_{i=1}^{n}\left\langle X_{i}v,v\right\rangle
X_{i}v\right\Vert ^{2}=\sum_{i,j}\left\langle X_{i}v,v\right\rangle
\left\langle X_{j}v,v\right\rangle \left\langle X_{i}v,X_{j}v\right\rangle 
\]

\begin{theorem}
Let $\{X_{1},...,X_{n}\}$ be a set of self adjoint elements of $End(V)$such
that $[X_{i},X_{j}]=0$ for $1\leq i,j\leq n$. There exists a constant $C>0$
such that of $v\in V$ then%
\[
C\left( \sum_{i,j}\left\langle X_{i}v,v\right\rangle \left\langle
X_{j}v,v\right\rangle \left\langle X_{i}v,X_{j}v\right\rangle \right)
^{2}\geq \left( \sum_{i}\left\langle X_{i}v,v\right\rangle ^{2}\right) ^{3}. 
\]
\end{theorem}

\begin{proof}
Let $S$ be the unit sphere in $V$. We note that the Theorem follows from the
following local version.

($\ast $) If $v_{o}\in S$ then there exists a neighborhood $\Omega _{v}$ of $%
v$ in and $C_{v}$ such that 
\[
C_{v}\left( \sum_{i,j=1}^{n}\left\langle X_{i}x,x\right\rangle \left\langle
X_{j}x,x\right\rangle \left\langle X_{i}x,X_{j}x\right\rangle \right)
^{2}\geq \left( \sum_{i=1}^{n}\left\langle X_{i}x,x\right\rangle ^{2}\right)
^{3},x\in \Omega _{v}. 
\]

Indeed, since $S$ is compact we can choose a finite number $%
v_{1},...,v_{r}\in S$ such that $\cup \Omega _{v_{i}}$ cover $S$. Choose $%
C=\max_{1\leq i\leq r}C_{v_{i}}.$

We will now prove $(\ast )$ by induction on $n$. If $n=1$ then we write $X$
for $X_{1}$ and we may assume that $X$ is diagonal. If $X=0$ then the
theorem is obvious. So assume $X\neq 0$ then we may take an orthonormal
basis $v_{1},...,v_{N}$ of $V$ such that $Xv_{i}=a_{i}v_{i}$ with $a_{i}\in 
\mathbb{R}$, $a_{i}\neq 0$ for $i=1,...,k$ and $a_{i}=0$ for $i>k$ and $%
|a_{i}|\geq \left\vert a_{i+1}\right\vert $. Now if $v=\sum x_{i}v_{i}$ then%
\[
\left\langle Xv,v\right\rangle \left\langle Xv,v\right\rangle \left\langle
Xv,Xv\right\rangle =\left\langle Xv,v\right\rangle ^{2}\sum
a_{i}^{2}x_{i}^{2}\geq 
\]%
\[
a_{k}^{2}\left\langle Xv,v\right\rangle ^{2}\sum_{i=1}^{k}x_{i}^{2}\geq 
\frac{a_{k}^{2}}{|a_{1}|}\left\langle Xv,v\right\rangle ^{2}\sum \left\vert
a_{i}\right\vert x_{i}^{2}\geq \frac{a_{k}^{2}}{|a_{1}|}\left\langle
Xv,v\right\rangle ^{2}\left\vert \left\langle Xv,v\right\rangle \right\vert
. 
\]%
This proves the theorem for $n=1$ hence $(\ast )$ in this case.

Now we assume that $(\ast )$. is true for $1\leq k<n$ and we prove it for $n$%
. If $\cap \ker X_{i}\neq (0)$ then the theorem follows from the case when $%
V $ is replaced by $Z=\left( \cap \ker X_{i}\right) ^{\bot }$ and the $X_{i}$
are replaced by $X_{i|Z}$. Thus we may assume that $\cap \ker X_{i}=(0)$. We
are now ready to prove the inductive step. Consider $v_{o}\in S$.

Let $B(v)$ denote the $n\times n$ matrix with $i,j$ entry $\left\langle
X_{i}v,X_{j}v\right\rangle $. Suppose that $v_{o}\in V$ is such that $%
X_{1}v_{o},...,X_{n}v_{0}$ are linearly independent. Then $B(v_{o})$ is
positive definite. Thus there is a compact neighborhood, $U$, of $v_{0}$in $S
$ and $C_{1}>0$ such that $B(v_{0})-C_{1}I$ is positive semidefinite. Thus
on $U$ we have%
\[
\sum_{i,j}\left\langle X_{i}v,v\right\rangle \left\langle
X_{j}v,v\right\rangle \left\langle X_{i}v,X_{j}v\right\rangle \geq
C_{1}\sum_{i}\left\langle X_{i}v,v\right\rangle ^{2}\text{.}
\]%
We note that there is a positive constant $C_{2}$ such that if $v\in S$ then 
$\left\vert \left\langle X_{i}v,v\right\rangle \right\vert \leq
C_{2}\left\langle v,v\right\rangle =C_{2}$. So%
\[
\left( \sum_{i}\left\langle X_{i}v,v\right\rangle ^{2}\right) ^{\frac{1}{2}%
}\leq \sqrt{n}C_{2}.
\]%
Thus on $U$ we have%
\[
\sum_{i,j}\left\langle X_{i}v,v\right\rangle \left\langle
X_{j}v,v\right\rangle \left\langle X_{i}v,X_{j}v\right\rangle \geq \frac{%
C_{1}}{\sqrt{n}C_{2}}\left( \sum_{i}\left\langle X_{i}v,v\right\rangle
^{2}\right) ^{\frac{3}{2}}.
\]%
The desired inequality. We may thus assume that the span of $\left\{
X_{i}v_{o}\right\} _{i=1}^{n}$ has dimension $1\leq l<n$.

Let $A=[a_{ij}]\in O(n)$ be as in the corollary above for $v_{o}$. We note
that $\sum_{i,j}\left\langle X_{i}v,v\right\rangle \left\langle
X_{j}v,v\right\rangle \left\langle X_{i}v,X_{j}v\right\rangle $ and $%
\sum_{i}\left\langle X_{i}v,v\right\rangle ^{2}$ are unchanged under the
transformation $X_{i}\rightarrow \sum a_{ij}X_{j}.$Replacing $X_{j}$ with $%
\sum_{i}a_{ji}X_{i}$ we may assume that if $l=\dim
Span\{X_{1}v_{o},...,X_{n}v_{o}\}$ then $X_{i}v_{o}=0$ for $i>l$ and the $%
X_{i}v_{o}$ for $i\leq l$ are mutually orthogonal. We come now to the only
place we use the assumption that $[X_{i},X_{j}]=0$ for $1\leq i,j\leq n$.

Let $\mathcal{A}$ denote the algebra generated by the $X_{i}$. Let $V_{0}=%
\mathcal{A}v_{o}$ and let $P:V\rightarrow V_{0}$ be the orthogonal
projection. Then we note that $X_{i}P=PX_{i}$ all $i$ and $X_{i}P=0$ if $i>l$%
. Now 
\[
\left\Vert \sum_{i=1}^{n}\left\langle X_{i}v,v\right\rangle
X_{i}v\right\Vert \geq \left\Vert \sum_{i=1}^{n}\left\langle
X_{i}v,v\right\rangle PX_{i}v\right\Vert = 
\]%
\[
\left\Vert \sum_{i=1}^{n}\left\langle X_{i}v,v\right\rangle
X_{i}Pv\right\Vert =\left\Vert \sum_{i=1}^{l}\left\langle
X_{i}v,v\right\rangle X_{i}Pv\right\Vert 
\]%
Noting that%
\[
\left[ \left\langle X_{i}v_{o},X_{j}v_{o}\right\rangle \right] _{1\leq
i,j\leq l}=\left[ \left\langle X_{i}Pv_{o},X_{j}Pv_{o}\right\rangle \right]
_{1\leq i,j\leq l} 
\]%
is positive definite we see that there exists $U$ be a compact neighborhood
of $v_{o}$ such that%
\[
B_{1}(\nu )=\left[ \left\langle X_{i}Pv,X_{j}Pv\right\rangle \right] _{1\leq
i,j\leq l} 
\]%
is positive definite for $v\in U$. We also note that we can choose a perhaps
smaller neighborhood such that 
\[
B_{2}(\nu )=\left[ \left\langle X_{i}v,X_{j}v\right\rangle \right] _{1\leq
i,j\leq l} 
\]%
is also positive definite for $\nu \in U$. Thus there is a constant $C_{3}>0$
such that $B_{1}(v)-C_{3}B_{2}(\nu )$ is positive semidefinite for $v\in U$.
So
\end{proof}

\[
\sum_{i,j=1}^{n}\left\langle X_{i}v,v\right\rangle \left\langle
X_{j}v,v\right\rangle \left\langle X_{i}v,X_{j}v\right\rangle \geq
\sum_{i,j=1}^{l}\left\langle X_{i}v,v\right\rangle \left\langle
X_{j}v,v\right\rangle \left\langle X_{i}Pv,X_{j}Pv\right\rangle \geq
C_{3}\sum_{i,j=1}^{l}\left\langle X_{i}v,v\right\rangle \left\langle
X_{j}v,v\right\rangle \left\langle X_{i}v,X_{j}v\right\rangle 
\]%
Set $C_{4}=\frac{1}{C_{5}}$. We have shown that if $v\in U$ then 
\[
C_{4}\left\Vert \sum_{i=1}^{n}\left\langle X_{i}v,v\right\rangle
X_{i}v\right\Vert \geq \left\Vert \sum_{i=1}^{l}\left\langle
X_{i}v,v\right\rangle X_{i}v\right\Vert 
\]%
There are obviously two possibilities for every $v\in S$

I. $2\left\Vert \sum_{i=1}^{l}\left\langle X_{i}v,v\right\rangle
X_{i}v\right\Vert \geq \left\Vert \sum_{i=l+1}^{n}\left\langle
X_{i}v,v\right\rangle X_{i}v\right\Vert $ or

II. $2\left\Vert \sum_{i=1}^{l}\left\langle X_{i}v,v\right\rangle
X_{i}v\right\Vert <\left\Vert \sum_{i=l+1}^{n}\left\langle
X_{i}v,v\right\rangle X_{i}v\right\Vert .$

We write $a=\left\Vert \sum_{i=1}^{l}\left\langle X_{i}v,v\right\rangle
X_{i}v\right\Vert ,b=\left\Vert \sum_{i=l+1}^{n}\left\langle
X_{i}v,v\right\rangle X_{i}v\right\Vert .$

We assume that $v\in U$. In case I. Observing that if $a,b\geq 0$ and $%
2a\geq b$ then $3a=a+2a\geq a+b$ thus in case I,%
\[
3C_{4}\left\Vert \sum_{i=1}^{n}\left\langle X_{i}v,v\right\rangle
X_{i}v\right\Vert \geq \left( \left\Vert \sum_{i=1}^{l}\left\langle
X_{i}v,v\right\rangle X_{i}v\right\Vert +\left\Vert
\sum_{i=l+1}^{n}\left\langle X_{i}v,v\right\rangle X_{i}v\right\Vert \right) 
\]%
and in case II. We have%
\[
\left\Vert \sum_{i=1}^{n}\left\langle X_{i}v,v\right\rangle
X_{i}v\right\Vert \geq \left\Vert \sum_{i=l+1}^{n}\left\langle
X_{i}v,v\right\rangle X_{i}v\right\Vert -\left\Vert
\sum_{i=1}^{l}\left\langle X_{i}v,v\right\rangle X_{i}v\right\Vert .
\]%
This time $b\geq 2a$ then%
\[
b-a\geq \frac{1}{3}b+\left( \frac{2}{3}b-a\right) =
\]%
\[
\frac{1}{3}b+\frac{1}{6}b\geq \frac{1}{3}\left( a+b\right) .
\]%
Thus in case II. We have%
\[
\left\Vert \sum_{i=1}^{n}\left\langle X_{i}v,v\right\rangle
X_{i}v\right\Vert \geq \frac{1}{3}\left( \left\Vert
\sum_{i=1}^{l}\left\langle X_{i}v,v\right\rangle X_{i}v\right\Vert
+\left\Vert \sum_{i=l+1}^{n}\left\langle X_{i}v,v\right\rangle
X_{i}v\right\Vert \right) .
\]%
Thus if $C_{5}$ is the maximum of $3$ and $3C_{4}$ we have for all $v\in U$ 
\[
C_{5}\left\Vert \sum_{i=1}^{n}\left\langle X_{i}v,v\right\rangle
X_{i}v\right\Vert \geq \left\Vert \sum_{i=1}^{l}\left\langle
X_{i}v,v\right\rangle X_{i}v\right\Vert +\left\Vert
\sum_{i=l+1}^{n}\left\langle X_{i}v,v\right\rangle X_{i}v\right\Vert 
\]%
Since $0<l<n$ the inductive hypothesis implies that there is an open
neighborhood $W$ of $v_{o}$ in $U$ and a constant $C_{6}>0$ such that%
\[
\left\Vert \sum_{i=1}^{l}\left\langle X_{i}v,v\right\rangle
X_{i}v\right\Vert +\left\Vert \sum_{i=l+1}^{n}\left\langle
X_{i}v,v\right\rangle X_{i}v\right\Vert \geq 
\]%
\[
C_{6}\left( \left( \sum_{i=1}^{l}\left\langle X_{i}v,v\right\rangle
^{2}\right) ^{\frac{3}{4}}+\left( \sum_{i=l+1}^{n}\left\langle
X_{i}v,v\right\rangle ^{2}\right) ^{\frac{3}{4}}\right) .
\]%
Thus for $v\in W$ we have%
\[
C_{5}\left( \sum_{i,j}\left\langle X_{i}v,v\right\rangle \left\langle
X_{j}v,v\right\rangle \left\langle X_{i}v,X_{j}v\right\rangle \right) ^{2}=
\]%
\[
C_{5}\left\Vert \sum_{i=1}^{n}\left\langle X_{i}v,v\right\rangle
X_{i}v\right\Vert ^{4}\geq C_{6}^{4}C_{5}\left( \left(
\sum_{i=1}^{l}\left\langle X_{i}v,v\right\rangle ^{2}\right) ^{\frac{3}{4}%
}+\left( \sum_{i=l+1}^{n}\left\langle X_{i}v,v\right\rangle ^{2}\right) ^{%
\frac{3}{4}}\right) ^{4}\geq 
\]%
\[
C_{6}^{4}C_{5}\left( \left( \sum_{i=1}^{l}\left\langle X_{i}v,v\right\rangle
^{2}\right) ^{3}+\left( \sum_{i=l+1}^{n}\left\langle X_{i}v,v\right\rangle
^{2}\right) ^{3}\right) .
\]%
We note that if $a,b\geq 0$ then $a^{3}+b^{3}\geq \frac{1}{8}(a+b)^{3}$. We
may assume $a\leq b.$ Then if $a=0$ the inequality is obvious so assume that 
$0<a\leq b$. Set $x=\frac{b}{a}\geq 1$ then $8+8x^{3}>1+3x^{3}+3x^{3}+x^{3}%
\geq 1+3x+3x^{2}+x^{2}=(1+x)^{3}$. Thus%
\[
C_{5}\left\Vert \sum_{i=1}^{n}\left\langle X_{i}v,v\right\rangle
X_{i}v\right\Vert ^{4}\geq \frac{C_{5}C_{6}^{4}}{8}\left(
\sum_{i=1}^{l}\left\langle X_{i}v,v\right\rangle
^{2}+\sum_{i=l+1}^{n}\left\langle X_{i}v,v\right\rangle ^{2}\right) ^{3}
\]%
for $v\in W$. This completes the induction.

\begin{center}
\bigskip

{\large Bibliography}

\bigskip
\end{center}

\noindent\lbrack H] Lars H\"{o}rmander, The Analysis of Linear Partial
Differential Operators II, Differential Operators with Constant Coefficients,
Springer-Verlag, Berlin, 1983, Appendix A, 362--371.\smallskip

\noindent\lbrack KN] George Kempf and Linda Ness, The length of vectors in
representation spaces, Algebraic geometry (Proc. Summer Meeting, Univ.
Copenhagen, 1978), Lecture Notes in Math. 732, Berlin, New York:
Springer-Verlag,1979 233--243\smallskip

\noindent\lbrack L] S. Lojasiewicz, Ensembles semi-analytiques, Preprint IHES,
1965.\smallskip

\noindent\lbrack N] Amnon Neeman, The topology of quotient varieties, Annals
of Math. (122),1985, 419--459.\smallskip

\noindent\lbrack RS] R. W. Richardson and P.\ J. Slodowy, Minimum Vectors for
Real Reductive Algebraic Groups, J. London Math. Soc. (2) 92 (1990),
409--429.\smallskip

\noindent\lbrack S] Gerald W. Schwarz, Topological methods in algebraic
transformation groups, 135-151, Progress in Mathematics, Volume 80,
Birkh\"{a}user, Boston, 1989.\smallskip

\noindent\lbrack W] Nolan R. Wallach, Geometric invariant theory over the
real and complex numbers, to appear, Springer.\smallskip

\end{document}